\documentclass[a4paper,leqno,12pt, reqno]{amsart}

\usepackage{amsmath}
\usepackage{amssymb}
\usepackage{graphicx}
\usepackage{tikz-cd}
\usepackage{hyperref}
\usepackage[all]{xy}

\usepackage{xspace}
\usepackage{bm}
\usepackage{amsmath}
\usepackage{amstext}
\usepackage{amsfonts}
\usepackage[mathscr]{euscript}
\usepackage{amscd}
\usepackage{latexsym}
\usepackage{amssymb}
\usepackage{enumerate}
\usepackage{xcolor}

\theoremstyle{plain}
\swapnumbers

\newcommand{\tX}{\widetilde{X}}

\newcommand{\tU}{\widetilde{U}}
\newcommand{\ZZ}{\mathbb{Z}}
\newcommand{\CC}{\mathbb{C}}
\newcommand{\DD}{\mathcal{D}}
\newcommand{\OO}{\mathcal{O}}

\newcommand{\tTC}[2]{\widetilde{\TC}_{#1}{{(#2)}}}

\newcommand{\dd}{\DD}
\newcommand{\tcd}[1]{\TC^{\dd}_{#1}}
\newcommand{\tcs}[2]{\TC^{\ast}_{{#1}}({#2})}
\newcommand{\tcg}[3]{\TC^{\ast}_{{#1}, {#2}}({#3})}
\newcommand{\tce}[3]{\TC_{{#1}, {#2}} ({#3})}

\newcommand{\xpi}{{\prod_\pi}\tX}

 \newtheorem{theorem}{Theorem}[section]
    \newtheorem{prop}[theorem]{Proposition}

    \newtheorem{lemma}[theorem]{Lemma}
    \newtheorem{corollary}[theorem]{Corollary}
       
     \newtheorem*{thma}{Theorem A}
    \newtheorem*{thmb}{Theorem B}
    \newtheorem*{thmc}{Theorem C}

\newenvironment{mysubsection}[2][]
{\begin{subsec}\begin{upshape}\begin{bfseries}{#2.}
\end{bfseries}{#1}}
{\end{upshape}\end{subsec}}

\theoremstyle{definition}
    \newtheorem{definition}[theorem]{Definition}
    \newtheorem{example}[theorem]{Example}

    \newtheorem{notation}[theorem]{Notation}
    
    \newtheorem{ack}[theorem]{Acknowledgements}
      \newtheorem{subsec}[theorem]{}

\theoremstyle{remark}
        \newtheorem{remark}[theorem]{Remark}

\newcommand{\cd}{\operatorname{cd}}

\newcommand{\map}{\operatorname{map}}
\newcommand{\cat}{\operatorname{cat}}

\newcommand{\Id}{\operatorname{Id}}
\newcommand{\inc}{\operatorname{inc}}
\newcommand{\TC}{\operatorname{TC}}
\newcommand{\secat }{\operatorname{secat }}
\newcommand{\genus }{\operatorname{genus }}

\usepackage{mathtools}
\DeclarePairedDelimiter\ceil{\lceil}{\rceil}

\begin{document}

\title[An upper bound for higher topological complexity] {An upper bound for higher topological complexity and higher strongly equivariant complexity}
\author{Amit Kumar Paul }
\author{Debasis Sen}

\address{Department of Mathematics and Statistics\\
Indian Institute of Technology, Kanpur\\ Uttar Pradesh 208016\\India}
\email{kamitp@iitk.ac.in, amitkrpaul23@gmail.com}

\address{Department of Mathematics and Statistics\\
Indian Institute of Technology, Kanpur\\ Uttar Pradesh 208016\\India}
\email{debasis@iitk.ac.in}

\date{\today}


\begin{abstract}
We prove an upper bound of higher topological complexity $\TC_n(X)$ using higher $\DD$-topological complexity $\tcd{n}(X)$ of a space $X$.  An intermediate invariant $\tTC{n}{X}$ is used in the proof. We interpret this invariant $\tTC{n}{X}$  as  higher analogue of strongly equivariant topological complexity of the universal cover of $\tX$ with the action of the fundamental group of $X$.  
\end{abstract}
\maketitle

\begin{section}{Introduction}

The \emph{topological complexity } $\TC(X)$ of a path connected space $X$ was introduced by Farber (see \cite{far}).   It is a measure of the complexity to construct a motion-planning algorithm on the space $X$. Let $I = [0,1]$ and $PX =X^I $ denotes the free path space. Consider the fibration

\begin{equation}\label{pi}
  p: PX \rightarrow X\times X,~~ \gamma \mapsto (\gamma(0), \gamma(1)).
 \end{equation}
 
\noindent Then $\TC(X)$ is defined to be the least positive integer $k$ such that there exists an open cover $\{U_1, \cdots, U_k\}$ of $X\times X$ with continuous section of $p$ over each $U_i$ (i.e. a continuous map $s_i : U_i \rightarrow PX$ satisfying $\pi \circ s_i = \Id_{U_i}$ for $i = 1, 2, \cdots, k$). Generalising the idea, Rudyak defined higher topological complexity (see \cite{rud}). He introduced \emph{n-th topological complexity } $\TC_n(X), \ n \geq 2$  such that $\TC_2(X) = \TC(X)$. We recall the definition of higher topological complexity in the next section.  It is well known that $\TC_n(X)$ is homotopy invariaint. Therefore one can define topological complexity of a discrete group $\pi$ as $\TC_n(\pi) = \TC_n(K(\pi,1)),$ where $K(\pi,1)$ is a Eilenberg- MacLane space with fundamental group $\pi$ and other homotopy groups trivial.

Computation of  topological complexity is difficult. With a few known exact computations of these invariants, there has been work to get better bounds of these numbers.  Using cohomological dimension $\cd(\pi)$ of the fundamental group $\pi = \pi_1(X)$, A. Costa and M. Farber (\cite{cf}) obtained the following upper bound for a finite cell complex $X$:~
$$\TC(X) \leq 2 \cd(\pi) + \dim X +1.$$Further we know $\cd(\pi) +1 \leq \TC(\pi) \leq 2 \cd(\pi)+1$. In (\cite{alex}) A. Dranishnikov improves this an upper bound to
$\TC(X)\leq \TC(\pi) +\dim X$. Later in (\cite{farber}) the authors introduced a $\DD$-topological complexity $\tcd{}(X)$ which has the property $ \tcd{}(X)\leq \TC(\pi)$. They showed that $\TC(X) \leq \tcd{}(X) + \ceil{\frac{2  \dim X - r}{r+1}},$ where $r$ is the connectivity of the universal cover of $X$.  Clearly this gives a better bound. We generalise the result for higher $\TC_n(X)$.

\begin{thma}

Let $X$ be a finite dimensional simplicial complex such that its universal cover $\tX$ is $r$-connecetd.  Then we have ,

$$\TC_n(X) \leq \tcd{n}(X) + \ceil[\Bigg]{\frac{n  \dim X - r}{r+1}}, ~~n\geq 2.$$

\end{thma}
(See Theorem \ref{theorema}).
 
\noindent In particular, for $r=1$, we obtain $\TC_n(X) \leq \TC_n(\pi) + \ceil{\frac{n  \dim X - 1}{2}} $ which a generalisation of \cite[Theorem 3.3]{alex} of A. Dranishnikov  (cf. Corollary \ref{cbound}) .  He used strongly equivariant topological complexity to prove the result. Note that there are other versions of equivariaint topological complexity (cf. \cite{ColG, LubM, BlaK}) all of which differ slightly from each other.  As in (\cite{farber}), to prove Theorem A  an intermediate invariant $\tTC{n}{X}$ is introduced. We introduce higher analogue of strongly equivariant complexity.  Generalising  \cite[Proposition 3.8 ]{farber}, we prove that $\tTC{n}{X}$ can be viewed as the higher strongly equivariant complexity $\tcg{n}{\pi}{\tX}$ of the universal cover $\tX$ with the action of the fundamental group $\pi$. 
\begin{thmb}

For any finite simplicial complex $X$, we have $$\tTC{n}{X} = \tcg{n}{\pi}{\tX},~~n\geq 2,$$ where $\tX $ be the universal covering and $\pi = \pi_1(X)$. 

\end{thmb}

(See Theorem \ref{ttcg})

\noindent The upper bound of $\TC$ in \cite{alex} was deduced by showing that $\TC(E)\leq \TC(B)+ \tcg{}{G}{F} -1$  for a fiber bundle $E\to B$ with fiber $F$ and structure group $G$. We prove a similar result for higher topological complexity. 

\begin{thmc}
Let $E, B$ be two locally compact metric spaces and  $E\to B$ be a fiber bundle with fiber $F$ and structure group $G$ acting properly on $F$. Then 
$$\TC_n(E)\leq \TC_n(B)+ \tcg{n}{G}{F} -1, ~~~n\geq 2.$$
\end{thmc}

(See Theroem \ref{tfhtc}. )

\subsection*{Organisation}
The organisation of the rest of the paper is as follows: In Section 2, we recall some basic definitions related to topological complexity and LS category. In Section 3 we recall the definition of higher $\DD$-topological complexity and prove some general properties. In section 4, we introduce the invariant $\tTC{n}{X}$ and use it to prove Theorem A. In the last section, we introduce higher strongly equivariant complexity, obtain its properties, and prove Theorem B,  Theorem C.

\begin{ack}
The first author was supported by PhD research fellowship of Indian Institute of Technology, Kanpur. 
\end{ack}

\end{section}

\begin{section}{Preliminary}
Here we review basic concept of LS-category, topological complexity and higher topological complexity of a space $X$. We also recall some equivariant analogues. For details we refer to \cite{ sva, Col, ColG, far, cor, rud, Mar}.

\begin{mysubsection}{LS-category and topological complexity} 
Let $q : E \to B $ be a onto map, then the \emph{sectional category} of $q$ is denoted by \emph{secat$(q)$} and define as the minimal positive integer $k$ such that we have an open cover $\{U_i\}_{i = 1}^k$ for $B$ and on each open subset $U_i$ we have a continuous map $s_i : U_i \to E$ with $q \circ s_i : U_i \to B$ is homotopic to the inclusion $\Id_{U_i} : U_i \hookrightarrow B$. The map $s_i$ is called local section for $q$. If $q : E \to B$ is fibration then $\secat(q) =  \genus(q)$, where $genus(q)$ of the fibration $q$ is the minimal positive integer $k$ such that we have an open cover $\{U_i\}_{i = 1}^k$ for $B$ and on each open subset $U_i$ we have a continuous map $s_i : U_i \to E$ satisfying $q \circ s_i = \Id_{U_i} : U_i \hookrightarrow B$. We denote by $P_0X$ be the space of all paths in $X$ starts from some fix point (say $x_0$) and $PX = X^I$ be free path space of $X$. Consider the fibrations $$p_0 : P_0X \to X, ~~\gamma \to \gamma(1);~~p : PX \to X \times X, ~\alpha \to (\alpha(0), \alpha(1)).$$
 \begin{definition}
The \emph{Lusternik-Schnirelmann category} (\emph{LS-category})  of $X$ is defined as $\cat(X):=\genus(p_0)$. The \emph{topological complexity} of $X$ is $\TC(X) := \genus(p)$.
 \end{definition}

For $n\geq 2,$ let $I_n$ denotes the wedge of $n$ intervals $[0, 1]_j, j = 1, 2, \cdots, n$, where $0_j \in [0, 1]_j$ are identified. Consider the mapping space $X^{I_n}$ and the fibration
\begin{equation}\label{een}
e_n : X^{I_n} \rightarrow X^n ,~~~e_n(\alpha) = (\alpha(1_1), \alpha(1_2),\cdots,\alpha(1_n)).
\end{equation} 

\noindent The \emph{n-th topological complexity}  of $X$ is defined to be $\TC_n(X) :=\genus(e_n)$. It can be defined alternatively as  $\TC_n(X) = $ genus$(e'_n)$, where 
\begin{equation}\label{eenprime}
e'_n : X^{I} \rightarrow X^n ,  ~~~~e'_n(\alpha) = (\alpha(0), \alpha(\frac{1}{n-1}), \alpha(\frac{2}{n-1}),\cdots,\alpha(1)).
\end{equation}
 This is because $e_n$ and $e_n'$ are both fibrational replacement of the diagonal map $X\to X^n$. Clearly $\TC_2(X)$ is nothing but $\TC(X)$.

Topological complexity is closely related to LS-category, satisfying the relation 
\begin{equation}\label{eqineq}
\cat(X^{n-1}) \leq \TC_n (X) \leq \cat (X^n) \leq \TC_{n+1} (X).
\end{equation}
\noindent It is clear from the above inequality that $\{\TC_n(X)\}$ is a non-decreasing sequence. If a space $Y$ is homotopy equivalent to $X$, then $\TC_n (Y) = \TC_n (X)$ for any $n \geq 2$. Consequently, $X$ is contractible  if and only if $ \TC_n(X) = 1$ for any $n \geq 2$.
\end{mysubsection}

\begin{mysubsection}{Equivariant LS-category} Throughout the paper, $G$ will denote a discrete group.  A topological space $X$ with an action of a group $G$ is called a $G$-space.  A continuous map $\phi: X\to Y$ between $G$-spaces is called a $G$-map (or an equivariant map) if $\phi(gx)= g\phi(x)$ for all $g\in G$ and $x\in X$. The set $\mathcal{O}(x) = \{ gx \ ; g \in G \}$ is called the \emph{orbit} of $x\in X$ and $G_x = \{ g \in G ; gx = x \}$ is called the \emph{isotropy group} at $x$. 
For a subgroup $H$ of $G$, the $H$-\emph{fixed point set} of $X$ is given by $$X^H = \{x \in X ; hx = x \text{ for all } h \in H \}.$$
We call $X$ is \emph{$G$-connected} if the $H$-fixed point set $X^H$ is path-connected for every subgroup $H$ of $G$.  Let $Y$ be an another $G$-space and $\phi, \psi : X \to Y$ be two $G$-maps. Then $\phi$ is said to be $G$-homotopic to $\psi$, written as $\phi \simeq_G \psi$, if there is a $G$-map $F : X \times I \to Y$ with $F(x, 0) = \phi(x)$ and $F(x ,1) = \psi(x)$, where $G$ acts trivially on $I$ and diagonally on $X \times I$. Two $G$-spaces $X,Y$ are called $G$-homotopy equivalent if there are  $G$-maps $\phi: X\to Y$ and $\phi': Y\to X$ such that $\phi\circ\phi' \simeq_G \Id_Y$ and $\phi'\circ\phi \simeq_G \Id_X.$

A subset $U\subset X$ is called \emph{G-invariant}
if $gU \subseteq U$ for all $g \in G$.  Such a $U\subseteq X$ is called \emph{$G$-categorical} if there exists a $G$-homotopy $F: U\times I \to X$ such that $F(-,0)$ is the inclusion map $U \hookrightarrow X$ and Image$(F(-,1)) \subset \OO(x')$ for some $x' \in X$. We say $X$ is \emph{$G$-contractible} if $X$ is $G$-categorical. 

\begin{definition}
The \emph{G-equivariant LS\mbox{-}category} of $X$, denoted by $\cat_G(X),$
 is the minimum positive integer $k$ such that $X$ can be covered by $k$ open sets $\{U_1, U_2, \cdots, U_k\}$, each of which is $G$-categorical.  
 \end{definition}
 
\noindent Clearly $X$ is $G$-contractible if and only if $\cat_G(X) = 1$. The $G$-equivariant category $\cat_G(X)$ is a $G$-homotopy invariant. Therefore if $X$ is $G$-homotopy equivalent to a point then $\cat_G(X)=1$, i.e. $X$ is $G$-contractible. 
The following lemma gives the converse implication. 

\begin{lemma}\label{lemma}
 For a $G$-connected $G$-space $X$ with $X^G \neq \phi$, the $G$-contractibility of $X$ implies $X$ is $G$-homotopy equivalent to a point (in $X^G$).

\end{lemma}

\begin{proof}

Since $X$ is $G$-contractible, there is a $G$-homotopy $F'_t : X \to X$ such that $F'_0 = \Id_X$ and $F'_1(x) \in \mathcal{O}(x')$ for some $x' \in X$. Consider  an element $x_0\in X^G$ and the isotropy group $H = G_{x'}$ at $x'$. Then the both elements $x', x_0 \in X^H$. Since $X$ is $G$-connected so $X^H$ is path-connected. Fix a path $\gamma : I \to X^H$, from $x'$ to $x_0$. Note that $H \subseteq G_{\gamma(t)}$ for all $t \in I$. Define a homotopy $F''_t : \mathcal{O}(x') \to X$ by $F''_t(g.x') = g.\gamma(t)$, where $g \in G$. Then $F''_t$ is well defined and $F''_0 = \Id_{\mathcal{O}(x')}$ , $F''_1(g.x') = x_0$. Define another homotopy $F_t = F'_t * F''_t: X \to X$, where $$F'_t * F''_t(x) = \begin{cases}F'_{2t}(x) &\mbox{if } 0 \leq t \leq \frac{1}{2} \\
F''_{2t-1}(F'_1(x)) &\mbox{if }\frac{1}{2} \leq t \leq 1 , \end{cases}$$ then $F_t$ is $G$-equivariant with $F_0 = \Id_X$ and $F_1(x) = x_0$ for all $x \in X$.

\end{proof}
\begin{example}\label{ecat}
Let $\ZZ_2$-acts on $S^n, n\geq 2$ by reflection. Then $\cat_{\ZZ_2}(S^n) = 2.$   Assume the reflection keeps the hyperplane perpendicular to $x_n$ fixed. Take $U: x_0 > -\frac{1}{2}\subset S^n $ and $V: x_0 < \frac{1}{2}\subset S^n$. Then usual contraction of $U,V$ are equivariant. So $U,V$ are $\ZZ_2$-categorical hence $\cat_{\ZZ_2}(S^n) \leq 2.$ If $\cat_{\ZZ_2}(S^n) = 1$, then by Lemma \ref{lemma} $S^n$ is $\ZZ_2$-homotopy equivalent to a point, and in particular homotopy equivalent to a point which is not true. So $\cat_{\ZZ_2}(S^n) =2.$
\end{example}

\noindent We will need the following lemma in later section. We refer to  \cite[Theorem 3.16]{ColG}, \cite[Theorem 2.23, Example 6.5]{sarkar1}, \cite[Proposition 2.29]{sarkar} for the proof of the lemma. 

\begin{lemma}\label{lcatg}
Suppose $Y_i$ is a $G_i$-space for $i=1,2$. Consider $Y_1\times Y_2$ as $G_1\times G_2$-space with the product action.  If $Y_1^{G_1}\neq \phi$ and $Y_2^{G_2}\neq \phi$ then $$\cat_{G_1\times G_2}(Y_1\times Y_2) \leq \cat_{G_1}(Y_1) + \cat_{G_2}(Y_2) -1.$$

\end{lemma}

\end{mysubsection}

\begin{mysubsection}{Deformable subset and $r$-cover} The topological complexity can be interpreted using deformable subsets. Let $A, U\subseteq X$. We call $U$ is \emph{$A$-deformable} if there is a homotopy $h_t : U \to X$ with $h_0 : U \hookrightarrow X$ is inclusion and $h_1(U) \subset A$.  An open cover $\mathcal{C} = \{U_1, U_2, \cdots, U_{r}\}$ is called \emph{$A$-deformable} if each $U_i$ is $A$-deformable. For a $G$-space $X$, let $A, U$ be invariant subsets. Then $U$ is called \emph{$A$-equivariantly deformable} if the above homotopy is an equivariant homotopy. It is known that the topological complexity $\TC_n(X)$ of a space $X$ is the minimum number $k$ such that there is a $\Delta(X)$-deformable open cover of $X^n$ containing $k$ open sets, where $\Delta(X) = \{(x, x,\cdots, x)\in X^n~; x\in X \}$. 

\noindent Now we recall some basic results about open covers which are described in (\cite{alex, kolmo, ostrand}). An open cover $\mathcal{C} = \{U_1, U_2, \cdots, U_{r+r'}\}$ of a space $X$ is called \emph{$r$-cover} if every subcollection of $r$ sets from $\mathcal{C}$ also a cover of $X$. We have the following simple observation.

\begin{lemma}\label{lemmacover}
 Let $\{U_1, U_2, \cdots, U_{r+r'-1}\}$ be an $r$-cover and  $\{V_1, V_2, \cdots, V_{r+r'-1}\}$  be an $r'$-cover of a space $X$, then  $\{W_1, W_2, \cdots, W_{r+r'-1}\}$  covers $X$ where $W_i = U_i \cap V_i$. 

\end{lemma}

\begin{proof}  Let $x\in X$.  By \cite[Proposition 2.1]{alex} an open cover $\mathcal{C} = \{U_1, U_2, \cdots, U_{r+r'-1}\}$ is an $r$-cover of a space $X$ if and only if each $x \in X$ is contained in at least $r' $ sets of $\mathcal{C}$.  Hence there is a subcollection $\{U_{i_1}, U_{i_2}, \cdots, U_{i_{r'}}\}$ of $\{U_1, U_2, \cdots, U_{r+r'-1}\}$ each of which contains $x$. Then the set  $\{V_{i_1}, V_{i_2}, \cdots, V_{i_{r'}}\}$ covers $X$. So $x\in V_{i_k}$ for some $k\in \{1, 2, \cdots r'\}$ and hence $x\in U_{i_k} \cap V_{i_k} = W_{i_k}$.

\end{proof}
We will need the following result in last section.
\begin{prop}\cite[Theorem 2.4]{alex} \label{propdeform} Assume $F$ is locally compact metric space and $A\subset F$. If $\{U'_i\}_{i=1}^r$ is a $A$-deformable open cover of $F$, then for any $r'\geq 0$ there is a $A$-deformable open $r$-cover $\{U_j\}_{j=1}^{r+r'}$ of $F$ such that $U_j = U'_i$ for $i=j \leq r$ and for $j > r, \ U_j =\sqcup_{i=1}^{r} V_{i}$ is a disjoint union with $V_i \subset U'_i$.

\noindent If $G$ acts on $F$ and $A$ is $G$-invariant and $\{U'_i\}_{i=1}^r$ is a $A$-equivariantly deformable open cover of $F$ by $G$-invariant sets, then for any $r'\geq 0$ there is a $A$-equivariantly deformable open $r$-cover $\{U_j\}_{j=1}^{r+r'}$ of $F$ by $G$-invariant sets such that $U_j = U'_i$ for $i=j \leq r$ and for $j > r, \ U_j = \sqcup_{i=1}^{r} V_{i}$ is a disjoint union with $G$-invariant subsets $V_i \subset U'_i$.

\end{prop}

\end{mysubsection}

\end{section}

\begin{section}{Higher $\dd$- topological complexity}

In \cite{farberm}, Farber, Grant, Lupton and Oprea introduce $\dd$-topological complexity for a path-connected space and proved that for a finite aspherical cell complex, the topological complexity and $\dd$-topological complexity are same. In \cite{farb}, Farber and Oprea define $n$-th $\dd$-topological complexity $\tcd{n}(X)$ and generalised the result. In \cite{farber}, Farber, Grant, Lupton and Oprea proved some properties of $\dd$-topological complexity. In this Section we generalised the results for $n$-th $\dd$-topological complexity.

\begin{definition}

Let $X$ be a path-connected topological space with fundamental group $\pi = \pi_1(X, x_0)$. The \emph{$n$-th $\dd$-topological complexity}, $\tcd{n}(X)$, is defined as the minimum number $k$ such that $X^n$ can be covered by $k$ open subsets, $X^n = U_1 \cup U_2 \cup \cdots \cup U_k$, with the property that for any $i = 1, 2,\cdots, k$ and for every choice of the base point $u_i \in U_i$, the homomorphism $\pi_1(U_i, u_i) \rightarrow \pi_1(X^n, u_i)$ induced by the inclusion $U_i \rightarrow X^n$ takes values in a subgroup conjugate to the diagonal $\Delta \subset \pi^n$, where $\pi^n = \pi \times \pi \times \cdots \times \pi \ (n$-times).

\end{definition}

\noindent We now interpret the above definition as sectional category of a certain covering map. 

\begin{prop}\label{pro}

Let $X$ be a path-connected, locally path-connected and semi-locally simply connected topological space with fundamental group $\pi = \pi_1(X, x_0)$. Let $q :\widehat{X^n} \rightarrow X^n$ be the connected covering space corresponding to the diagonal subgroup $\Delta \subset \pi^n = \pi_1(X^n, X_0)$, where $X_0 = (x_0, x_0, \cdots, x_0)$. Then 
\begin{center}
$\tcd{n}(X) = \secat (q).$
\end{center}
\end{prop}

\begin{proof}

Let $\secat (q) = k$, with $\{U_1, U_2, \cdots, U_k\}$ be cover of $X^n$ and for each $i, \ s_i : U_i \rightarrow \widehat{X^n}$ be a section on $U_i$ of $q$. Now by the lifting criterion of covering space $i_*(\pi_1(U_i, u_i)) \subseteq q_*(\pi_1(\widehat{X^n},\widehat{x_0}))$ (where \ $\widehat{x_0} \in \widehat{X^n}$), that is, $i_*(\pi_1(U_i, u_i)) \subseteq q_*(\Delta) $. Hence $\tcd{n}(X) \leq \secat (q)$.

Conversely, let $\tcd{n}(X) = k$. Then $i_*(\pi_1(U_i, u_i)) \subseteq$ some conjugate of $\Delta,$ means $i_*(\pi_1(U_i, u_i)) \subseteq q_*(\pi_1(\widehat{X^n},\widehat{x_0}))$. Again by lifting criterion of covering space, a lift $s_i : U_i \rightarrow \widehat{X^n}$ exist, i.e. section exist on $U_i$.  Hence $\tcd{n}(X) \geq \secat (q)$.
 
\end{proof}

\begin{example}

For a path-connected space $X$, $\tcd{n}(X) = 1$ if and only if $X$ is simply connected. So we have $\tcd{n}(S^m) = 1$ for all $m,n \geq 2$.

\end{example}

\noindent To get an analogue of Equation (\ref{eqineq}) for $\DD$-topological complexity, we recall the definition of Lusternik-Schnirelmann one-category.

\begin{definition}
Let $X$ be a path connected, locally path-connected and semi-locally simply connected space with universal cover $P : \tX \to X$. The \emph{Lusternik-Schnirelmann one-category} of $X$ is defined to be $\cat_1(X) := \secat(P)$ of $P$.
\end{definition}

\begin{prop}

If $X$ is a path-connected, locally path-connected  and semi-locally simply connected topological space, then

$$\cat_1(X^{n-1}) \leq \tcd{n}(X) \leq \cat_1(X^n).$$

\end{prop}

\begin{proof}
Consider the commutative diagram,

$$\xymatrix{
\overline{X} \ar[dd]^{q'} \ar[rr]  &&  \widehat{X^n} \ar[dd]^{q}  &&  \tX^n \ar[ll] \ar[dd]^{P^n}  \\\\
X^{n-1} \ar[rr]_{f} && X^n && X^n \ar[ll]^=
}$$
where $q : \widehat{X^n} \to X^n$ is the cover corresponds to the diagonal group $\Delta \subset \pi^n = \pi \times \pi \times \cdots \times \pi$. The map $f : X^{n-1} \to X^n$ is the inclusion to the first $n-1$ factor, $f(x_1, x_2,\cdots, x_{n-1}) = (x_1, x_2,\cdots, x_{n-1}, *)$, where $* \in X$ is the base point and $\overline{X}$ is the preimage $q^{-1}(f(X^{n-1}))$. Note that  $q_*(\pi_1(\widehat{X^n}))$ and $f_*(\pi_1(X^{n-1}))$ spans $\pi_1(X^n)$. So using property of pullback covering by inclusion map we can say that $\overline{X}$ is covering space corresponding to the subgroup $ f_*^{-1}(q_*(\pi_1(\widehat{X^n})) \cap f_*(\pi_1(X^{n-1})))$ which is trivial. Thus $q' : \overline{X} \to X^{n-1}$ is the universal cover of $X^{n-1}$.

Given an open subset $U \subset X^n$ with a section $s : U \to \widehat{X^n}$ we may restrict it to $f^{-1}(U) \subset X^{n-1}$ getting a section $s' : f^{-1}(U) \to \overline{X}$. This shows that $\cat_1(X^{n-1}) = \secat(q') \leq \secat(q) = \tcd{n}(X)$, thus proving the first inequality.

Next we consider the right square of the diagram. The map $P^n$ is the universal covering and hence $\secat(P^n) = \cat_1(X^n) \geq \secat(q) = \tcd{n}(X)$.  This is the second inequality. 
\end{proof}

\begin{corollary}\label{cineq}
For $X$ as above,   $$\cat_1(X^{n-1}) \leq \tcd{n}(X) \leq \cat_1(X^n) \leq \tcd{n+1}(X).$$ In particular $\tcd{n}(X) \leq \tcd{n+1}(X)$, for all $n \geq 2$.

\end{corollary}

Now we relate higher topological complexity $\TC_n(X)$ with higher $\dd$-topological complexity $\tcd{n}(X)$.

\begin{notation}\label{nota}
Let $P: \tX \to X$ be the universal cover of $X$. Let $\pi = \pi_1(X)$ denotes the fundamental group of $X$ and $\xpi$ stands for the quotient of ${\tX}^n$ with respect to the diagonal action of $\pi$.
\end{notation}

\begin{prop}

For a path-connected, locally path-connected and semi-locally simply connected topological space $X$ one has  $\tcd{n}(X) \leq \TC_n(X)$.

\end{prop}

\begin{proof}

 Consider the projection $q : \xpi \to X^n$. Clearly $q$ is a covering map with the property that the image of the induced homomorphism $q_* : \pi_1(\xpi) \to \pi_1(X^n)$ is the diagonal. Hence by Proposition \ref{pro}, $\tcd{n}(X) = \secat (q)$.

Now we define 
$$p : X^I \to \xpi , ~~~~~~~~~~~\gamma \mapsto [\widetilde{\gamma}(0), \widetilde{\gamma}(\frac{1}{n-1}), ...,  \widetilde{\gamma}(\frac{j}{n-1}), \cdots, \widetilde{\gamma}(1)],$$ where $\widetilde{\gamma} : I \to \tX$ is any lift of the path $\gamma : I \to X$ and the brackets $[x_0, x_1, \cdots, x_{n-1}]$ denote the orbit of the tuple $(x_0, x_1, \cdots,  x_{n-1}) \in {\tX}^n$ with respect to the diagonal action of $\pi$. The map $p$ is well-defined although the lift $\widetilde{\gamma}$ is not unique. We obtain the following commutative diagram.
$$
\xymatrix{
X^I \ar[rrrr]^{p} \ar[drr]_{e'_n}&&& &\xpi \ar[dll]^{q}\\
&& X^n & &
}
$$

\noindent Clearly, a partial section $s : U \to X^I$ of $e'_n$ gives a partial section $\widetilde{s} = p\circ s : U \to \xpi$ of $q$. So we have $\tcd{n}(X) = \secat(q) \leq \secat(e'_n) = \TC_n(X)$.
\end{proof}

For aspherical spaces, $\dd$-topological complexity is same as topological complexity.
\begin{lemma}\cite[Lemma 4.2 ]{farb}\label{pdtc} Let $X$ be an aspherical CW complex. Then 
\begin{center}
$\tcd{n}(X) = \TC_n(X)$.
\end{center}

\end{lemma}

Now we show that the $\tcd{n}{(X)}$ is also homotopy invariant.

\begin{prop}\label{pprop}

Assume that $f : X \rightarrow Y$ is a continuous map between path-connected topological spaces such that the induced map $f_* : \pi_1(X) \rightarrow \pi_1(Y)$ is an isomorphism. Then we have,

\begin{center}
$\tcd{n}(X) \leq \tcd{n}(Y)$.
\end{center}

\end{prop}
\begin{proof}

Let $U \subset Y^n$ be an open subset such that the induced homomorphism $\pi_1(U, u) \rightarrow \pi_1(Y^n, u)$ takes values in a subgroup conjugate to the diagonal. Consider the preimage $V = (f \times f \times \cdots \times f)^{-1}(U) \subset X^n$. The map $\pi_1(V) \to \pi_1(X^n),$ induced by the inclusion, can be factored as the composition $$\pi_1(V) \rightarrow \pi_1(U) \rightarrow \pi_1(Y^n) \xrightarrow{(f_*^{-1})^n} \pi_1(X^n).$$  Since the second map takes values in a subgroup conjugate to the diagonal, hence the map $\pi_1(V) \rightarrow \pi_1(X^n)$ also has the same property. Therefore $\tcd{n}(X) \leq \tcd{n}(Y)$.

\end{proof}

\begin{corollary}

The higher $\dd$-topological complexities are homotopy invariant.

\end{corollary}

\begin{proof}
Assume $f: X\to Y$ is a homotopy equivalence with inverse $g: Y\to X$.  Then applying the above proposition to $f$ and $g$ we get $\tcd{n}{(X)} = \tcd{n}{(Y)}.$

\end{proof}

\noindent Since $\tcd{n}{(X)}$ is homotopy invariant, we can define $\tcd{n}{(\pi)} :=\tcd{n}{(K(\pi,1))}$ for a discrete group $\pi$. Note that $\tcd{n}{(\pi)}  = \TC_n(\pi)$ by Lemma \ref{pdtc}. 
\begin{prop}\label{pdasp}

Let $X$ be a path-connected CW complex with fundamental group $\pi = \pi_1(X)$. Then 
$$\tcd{n}(X) \leq \tcd{n}(\pi).$$
Moreover, if there exist a positive integer $k \geq 2$ such that the homotopy groups $\pi_j(X) = 0$ for all $j$ satisfying $1 < j < k$ and $\pi$ has cohomological dimension $\leq k$, then  $$\tcd{n}(X) = \tcd{n}(\pi).$$

\end{prop}

\begin{proof}

The Eilenberg-Mac Lane complex $K = K(\pi, 1)$ can be constructed by attaching cells of dimension $\geq 3$ to $X$. Consider the inclusion map $i : X \hookrightarrow K$ which induces isomorphism of fundamental groups. So, using the Proposition \ref{pprop} we can say that $\tcd{n}(X) \leq \tcd{n}(K) = \tcd{n}(\pi)$.

For the second part, the Eilenberg-Mac Lane space $K = K(\pi, 1)$ can be obtained from $X$ by attaching cells of dimension $k+1, k+2, \cdots$. Now convert the inclusion $X \hookrightarrow K$ into a fibration with fiber $F$ satisfying $\pi_i(F) = \pi_{i+1}(K, X)$. Since $\pi_1(X) \simeq \pi_1(K)$, we have $\pi_i(F) = \pi_{i+1}(K, X) = 0$ for $i = 0, 1,\cdots, k-1$. The obstructions to finding a section of $X \rightarrow K$ lie in the groups $H^{i+1}(\pi, \pi_i(F)) = H^{i+1}(K, \pi_i(F))$ and all these groups vanish because our computation with $\pi_i(F)$ and our assumption $\cd(\pi) \leq K$. Finally if we apply Proposition \ref{pprop} to the section, which induces on the fundamental groups, we get $\tcd{n}(X) \geq \tcd{n}(K) = \tcd{n}(\pi)$. So $\tcd{n}(X) = \tcd{n}(\pi)$.

\end{proof}

We now show that if  $X$ has a group structure, then the left side inequality of Corollary \ref{cineq} is an equality.

\begin{prop}

For any connected topological group $H, \  \tcd{n}(H) = \cat_1(H^{n-1})$.

\end{prop}

\begin{proof}

Let $F : H^n \rightarrow H^{n-1}$ be the map given by the formula $$F(a_1, a_2,\cdots, a_n) = (a_1a_n^{-1}, a_2a_n^{-1},\cdots, a_{n-1}a_n^{-1}).$$ Denote $\pi = \pi_1(H, e)$ and consider the induced map on fundamental groups 

\begin{center}

$\phi = F_* : \pi^n = \pi_1(H^n, e^n) \longrightarrow \pi^{n-1} = \pi_1(H^{n-1}, e^{n-1})$

\end{center}
where $\pi^{n} = \pi \times \pi \times \cdots \times \pi, \ n$-times and $e^n = (e, e,\cdots, e)$ is $n$-tuple in $H^n$. From the definition of $F$, it is clear that $F_*(\alpha_1, \alpha_2,\cdots, \alpha_n) = (\alpha_1 - \alpha_n, \alpha_1 - \alpha_n, \cdots, \alpha_{n-1} - \alpha_n)$. Note that the kernel of $\phi$ is $\Delta \subset \pi^n$, the diagonal subgroup. This gives a pullback  diagram of covering maps
\begin{center}

$\xymatrix{
\widehat{H^n} \ar[dd]_{q} \ar[rr]^{\widetilde{F}}  &&  {\widetilde{H}}^{n-1} \ar[dd]^{P^{n-1}} \\\\
H^n \ar[rr]_{F} && H^{n-1}
}$

\end{center}
where $P : \widetilde{H} \to H$ is universal covering and $q$ is the covering corresponds to the diagonal subgroup. From the diagram we obtain $\tcd{n}(H) = \secat(q) \leq \secat(P^{n-1}) = \cat_1(H^{n-1})$.

\end{proof}

\end{section}

\begin{section}{The invariant $\tTC{{n}}{X}$ }
In this section we introduce an invaraint $\tTC{n}{X}$ which is higher analogue of $\tTC{}{X}$ as defined in \cite{farber}. We use it to prove Theorem (A). Consider maps $E \xrightarrow {p} \overline{X} \xrightarrow{q} X$, where $p$ is a fibration with fiber $F$, $q$ is a covering map with fiber $F_0$ and the space $\overline{X}$ is connected. The composition is a fibration with fiber $F'$ which is homeomorphic to $F \times F_0$. 

\begin{definition}
With notations as above, the number $\widetilde{\secat}(E \xrightarrow {p} \overline{X} \xrightarrow{q} X)$ is the minimal integer $k \geq 1$ such that $X$ admits an open cover $X = U_1 \cup U_2 \cup \cdots \cup U_k$, with the property that for each $1\leq i\leq k,$ the fibration $p$ admits a continuous section over the open set $q^{-1}(U_i) \subset \overline{X}$. 
\end{definition}
\noindent It is clear from the definition that $\widetilde{\secat}(E \xrightarrow {p} \overline{X} \xrightarrow{q} X) \geq \secat(p)$ and $\widetilde{\secat}(E \xrightarrow {p} \overline{X} \xrightarrow{q} X) = 1$ if and only if $\secat(p) = 1$. The following result is proved in \cite[Proposition 3.2]{farber}. Note that we are counting from $1$ in the definitions of genus and related things.

\begin{lemma}\label{lem}

With notations as above, we have $$\secat(q \circ p) \leq \secat(q) + \widetilde{\secat}(E \xrightarrow {p} \overline{X} \xrightarrow{q} X) -1.$$

\end{lemma}

\noindent Consider the quotient 

\begin{center}
$E = \{(\omega , x_1, x_2,\cdots , x_n);  \ \omega \in {\tX}^I, \omega (0)=x_1, \omega(\frac{1}{n-1})=x_2, \cdots, \omega(1)=x_n)\}/\pi$,
\end{center}
where $\pi = \pi_1(X)$. Note that $E$ can be identified with  $X^I$ by a choice of lift of path in $X$ to path in $\tX$. The quotient by the fundamental group ensures that this is well defined. Recall that $\xpi$ stands for the quotient of ${\tX}^n$ with respect to the diagonal action of $\pi$ (see Notation \ref{nota}). We define two maps $p,q$ as follows: 
 $$p : E \to { \prod_\pi} \tX,~~~~p([\omega , x_1, x_2,\cdots , x_n]) \mapsto [x_1, x_2,\cdots , x_n], $$ and $$q :  \prod_\pi \tX\ \to X^n,~~~q([x_1, x_2,\cdots , x_n]) \mapsto (Px_1, Px_2,\cdots , Px_n).$$ Here $P : \tX \to X$ is the universal cover. Now we have the situation 
$$
\xymatrix{
X^I \ar[rrr]^{p} &&& { \prod_\pi} \tX \ar[rrr]^{q}&&& X^n.}
$$
For such $p,q$, it is clear that $$\secat(p \circ q) = \TC_n(X),~~~~~~\secat(q) = \tcd{n}(X).$$

\begin{definition}

With notations as above, we define $$\tTC{n}{X} = \widetilde{\secat} (X^I \xrightarrow{p} { \prod_\pi} \tX \xrightarrow{q} X^n).$$

\end{definition}

Applying the Lemma \ref{lem} in our this particular case we have

\begin{equation}\label{ineq1}
\TC_n(X) \leq \tcd{n}(X) + \tTC{n}{X} -1.
\end{equation}

\begin{lemma}

For a CW-complex $X$, the following statements are equivalent.
\begin{enumerate}[(i)]
\item  For some $n\geq 2$, $\tTC{n}{X} = 1$.

\item The space $X$ is aspherical.

\item  For all $n\geq 2$, $\tTC{n}{X} = 1$.
\end{enumerate}

\end{lemma}

\begin{proof}

(i) $\Rightarrow$ (ii):

 \noindent Suppose that $\widetilde{TC}_n(X) = 1$, for some $n \geq 2$. Then the fibration $p : X^I \to \prod_\pi \tX$ has a continuous section. Now for $r \geq 2$ consider the composition 

$$\pi_r(X) = \pi_r(X^I) \xrightarrow{p_*} \pi_r(\prod_\pi \tX) \xrightarrow{\simeq} \pi_r(X^n) = \bigoplus \pi_r(X),$$
where $\bigoplus \pi_r(X)$ is the direct sum of $n$-copies of $\pi_r(X)$. Since $p$ has a section so this composition must surjective and it is possible only when $\pi_r(X) = 0$, for all $r \geq 2$. So $X$ is aspherical.

\noindent (ii) $\Rightarrow$ (iii):

 \noindent If $X$ is aspherical then $\tX$ is contractible. The fiber of $p : X^I \to \prod_\pi \tX$  is the mapping space $\map(\bigvee_{(n-1)\text{ copies} } S^1 \to  \tX)$, which is also contractible. This implies $p$ has a section and hence $\tTC{n}{X} = 1$ for all $n \geq 2$.

\noindent (iii) $\Rightarrow$ (i) is obvious.
\end{proof}

\begin{prop}

Let $Z = X \times Y$ where $X = K(\pi, 1)$ is aspherical and $Y$ is simply connected. Then $\tcd{n}(Z) = \TC_n(X)$ and $\tTC{n}{Z} = \TC_n(Y)$.

\end{prop}

\begin{proof}

The first equality follows from the Proposition \ref{pprop} applying on $X \to X \times Y \to X$, injection and projection. The proof of second equality is similar as (\cite{farber}, Proposition 3.11).

\end{proof}

We now use the higher $\dd$-topological complexity to give an upper bound of $\TC_n(X)$ using connectivity of the universal covering space $\tX$ of X. This is a generalisation of the result \cite[Theorem 4.3]{farber}.

\begin{theorem}\label{theorema}

Let $X$ be a finite dimensional simplicial complex such that its universal cover $\tX$ is $r$-connected.  Then we have ,
\begin{equation}\label{eq1}
\TC_n(X) \leq \tcd{n}(X) + \ceil[\Bigg]{\frac{n  \dim X - r}{r+1}} .
\end{equation}
In particular if $\tX$ is $ (n-1)$-connected, then

$$\TC_n(X) \leq \tcd{n}(X) +  \dim X .$$

\end{theorem}

\begin{proof}

If we have a covering map $q : \overline{B} \to B$, with $B$ finite dimensional simplicial complex and a fibration $p : E \to \overline{B}$ with $(r-1)$-connected fiber for some $r \geq 0$, then by (\cite[Theroem 4.1]{farber}) we have $$\widetilde{\secat}(E \xrightarrow {p} \overline{B} \xrightarrow{q} B) \leq \ceil{ \frac{\dim B - r}{r+1}} + 1.$$ We apply this result to the defining maps  $X^I \xrightarrow{p} { \prod_\pi} \tX \xrightarrow{q} X^n$ of $\tTC{n}{X}.$ The fiber of the map $p$ is the mapping space $F= \map (\bigvee_{(n-1)\text{ copies} }S^1 \to  \tX)$. Since $\tX$ is $r$-connected, we get that $F$ is $(r-1)$-connected. This can be seen by considering the fibration $P_0\tX \to \tX^{n}$ given by the projections at $0, \frac{1}{n-1} , \frac{2}{n-1},\cdots, \frac{n-1}{n-1}=1$. This also has fiber $F$. Looking at the homotopy long exact sequence gives us the desired connectivity of $F$.
So we get 
\begin{equation}\label{eqttc}
\tTC{n}{X}\leq  \ceil[\Bigg]{\frac{n  \dim X - r}{r+1}} + 1. 
\end{equation}

Combining with the Equation (\ref{ineq1}) we get the Equation (\ref{eq1}).

If $r = n-1$, then $\ceil{\frac{n  \dim X - r}{r+1}} = \ceil{\dim X - \frac{n-1}{n}} = \dim X,$ so we obtain

$$\TC_n(X) \leq \tcd{n}(X) +  \dim X.$$

\end{proof}
The following corollary is a generalisation of \cite[Theorem 3.3]{alex}.

\begin{corollary}\label{cbound}
For a finite dimensional simplicial complex $X$ with fundamental group $\pi$ we have $$\TC_n(X) \leq \TC_n(\pi) + \ceil[\Bigg]{\frac{n  \dim X - 1}{2}} .$$
\end{corollary}

\begin{proof}
Combining Proposition \ref{pdasp} and Proposition \ref{pdtc} we have 
$$\tcd{n}(X) \leq \tcd{n}(\pi) = \TC_{n}(\pi) .$$
Since universal cover $\tX$ is $1$-connected (simply connected), putting $r=1$ in the Equation (\ref{eq1}) we get the result. 
\end{proof}

\section{Higher strongly equivariant topological complexity}

In (\cite{ColG}) Colman and Grant introduced equivariant topological complexity for a $G$-space $Y$. It is denoted by $\TC_G(Y)$ and is defined as the minimum integer $k\geq 1$ such that there exist $G$-invariant open subsets $U_1, U_2, \cdots, U_k$ covering  $Y \times Y$ under the diagonal action of $G$ on $Y \times Y$ and on each open subset there is $G$-equivariant section of the path fibration map $p : Y^I  \to Y \times Y, \ \gamma \to (\gamma(0), \gamma(1))$. In (\cite{sarkar}) Bayeh and Sarkar generalized equivariant version to higher topological complexity. Dranishnikov (\cite{alex}) introduced strongly equivariant topological complexity $\tcs{G}Y$ of a $G$-space $Y$, in which the covering open subsets are $G \times G$ invariant and the sections $s_i : U_i \to Y^I$ are $G$-equivariant with diagonal action of $G$ on $U_i$. In this section we introduce higher strongly equivariant topological complexity $\tcg {n}{G}Y$ of a $G$-space $Y$ and obtain some properties. After that we relate $\tTC{n}{X}$ and $\tcg{n}{\pi}{\widetilde{X}}$, where  $\tX$ is the universal cover of $X$ and $\pi = \pi_1(X)$.

\begin{definition}\label{def}
For a $G$-space $Y$, consider $Y^n$ as a $G^n$-space with product action. Consider the fibration $e_n : Y^{I_n} \rightarrow Y^n $ (cf. Equation \ref{een}) for $Y$. The $G$-action on $Y$ naturally induces a $G$-action on $Y^{I_n}. $ We define the \emph{$n$-th strongly equivariant topological complexity} $\tcg{n}{G}{Y}$, as the the minimal number $k$ such that $Y^n$ can be cover by $G^n$-invariant open sets $\{U_i:~ i=1,\cdots, k\}$ and there is a $G$-equivariant continuous section $s_i : U_i \to Y^{I_n}$ of $e_n$ for $i = 1, \cdots, k$ (considering $G$ as the diagonal subgroup of $G^n$). If no such $k$ exist, then $\tcg{n}{G}Y = \infty$.

\end{definition}

\noindent The following lemma shows that we can take the fibration $e_n' : Y^I \to Y^n$ in the above definition (cf. Equation \ref{eenprime}). Later we will use them interchangeably.

\begin{lemma}

Let $Y$ be a $G$-space and $U$ be a $G^n$-invariant open set of $Y^n$, then admitting $G$-equivariant continuous section on $U$ of the maps $e_n : Y^{I_n} \to Y^n$ and $e'_n : Y^{I} \to Y^n$ are equivalent.

\end{lemma}

\begin{proof}

Let $s : U \to Y^{I_n}$ be a $G$-equivariant continuous section of $e_n$. Define a map $\phi : I \to I_n$ by $[\frac{2j-2}{2(n-1)},  \frac{2j-1}{2(n-1)}]$ goes to linearly on $[0 , 1]_j$ in the reverse direction and $[\frac{2j-1}{2(n-1)}, \frac{2j}{2(n-1)}]$ goes to linearly on $[0 , 1]_{j+1}$ for $j= 1, 2, \cdots, n-1$. This map induces $\phi^{*} : Y^{I_{n}} \rightarrow Y^{I}$ such that the following diagram commutes.

 $$
 \xymatrix{
Y^{I_{n}} \ar[rrrr]^{\phi^{*}}  \ar[ddrr]_{e_{n}}& &&&Y^{I} \ar[ddll]^{e'_{n}}\\
&& U \ar@{^{(}->}[d] \ar[llu]_s \ar@{.>}[rru]^{s'} \\ && Y^n 
}
$$

\noindent Consider the composition map $s' = \phi^* \circ s : U \to Y^I$. Let $g \in G$ and $(y_1, y_2, \cdots, y_n) \in U$. Using the fact  $s : U \to Y^{I_n}$ is $G$-equivariant, we have, 
\begin{equation*}
\begin{split}
\phi^* \circ s (gy_1, gy_2, \cdots, gy_n)(t)  =~ & \phi^* ( s (gy_1, gy_2, \cdots, gy_n))(t)  \\
=~ &   s(gy_1, gy_2, \cdots, gy_n) \circ \phi(t)  \\
=  ~&  g. (s (y_1, y_2, \cdots, y_n)) \circ \phi(t) 
\\ =~ &  g. \phi^* \circ s (y_1, y_2, \cdots, y_n)(t). 
\end{split}
\end{equation*}
So the map $s' = \phi^* \circ s$ is a $G$-equivariant section of $e'_n$.

Conversely, let $s' : U \to Y^I$ be a $G$-equivariant continuous section of $e'_n$. Define a map $\psi : I_n \to I$ by sending  $[0, 1]_j$($1 \leq j \leq n$) linearly to:
$$ [\frac{j-1}{n-1}, \frac{1}{2}] \ \text{in the reverse direction}, \ \ \ \ \ \ \ \ ~~\text{if} ~~j \leq \frac{n+1}{2}$$ 
$$ [\frac{1}{2}, \frac{j-1}{n-1}], \ \ \ \ \ \ \ \ \ \  \ \ \ \ \ \ \ \ \ \ \ \ \ \ \ \ \  \ \ \ \ \ \ \ \ \ \ \ \  ~~~ \text{if} ~~j > \frac{n+1}{2}.$$

 \noindent So the map $\psi$ induces $\psi^{*} : Y^{I} \rightarrow Y^{I_{n}}$ such that the following diagram commutes.
$$
 \xymatrix{
Y^{I} \ar[rrrr]^{\psi^{*}}  \ar[ddrr]_{e'_{n}}& &&&Y^{I_n} \ar[ddll]^{e_{n}}\\
&& U \ar@{^{(}->}[d] \ar[llu]_{s'} \ar@{.>}[rru]^{s} \\ && Y^n 
}
$$
 As in previous case the composition $s = \psi^* \circ s' : U \to Y^{I_n}$ will be a $G$-equivariant section of $e_n$.
\end{proof}

So in Definition \ref{def} we can take the fibration $e'_n$ instead of $e_n$. 
As in the other cases of higher topological complexity, the sequence $\{\tcg{n}{G}{Y}\}_{n\geq 2}$ is non-decreasing. 

\begin{prop}

 Let $Y$ be a $G$-space, then $\tcg{n+1}{G}{Y} \geq \tcg{n}{G}{Y}$, for any $n \geq 2$.

\end{prop}

\begin{proof}

Take the inclusion $Y^n \to Y^{n+1}$ to the first $n$-factors. Given  an open cover $\{U_i:~ i=1,\cdots, k\}$ of  $Y^{n+1}$ by $G^{n+1}$-invariant subsets, take $V_i = U_i \cap Y^n$. Then $V_i$ are $G^n$ invariant with $G^n \hookrightarrow G^{n+1}$ as $(g_1, \cdots, g_n) \mapsto (g_1, \cdots, g_n, e)$. Restrict the section over $U_i$ to $V_i$. This satisfies the desired properties.

\end{proof}

We now show that $n$-th strongly equivariant topological complexity is $G$-homotopy invariant.  
\begin{prop}

Let $X,Y$ be $G$-spaces and there are $G$-maps $\phi : X \to Y$, ~$\psi : Y\to X$ such that $\phi\circ\psi \simeq _G \Id_Y$. Then $\tcg{n}{G}{X} \geq \tcg{n}{G}{Y}$. In particular if $X$ is $G$-homotopy equivalent to $Y$ then $\tcg{n}{G}{X} = \tcg{n}{G}{Y}.$

\end{prop}

\begin{proof}

Take an $G^n$-invariant open set $U \subset X^n$ with an equivariant section $s: U \to X^I$ of $e_n'$. Consider $V= (\psi^n)^{-1} (U)\subset Y^n$. Then $V$ is also $G^n$-invariant. 
$$\xymatrix{
& Y^I \ar[dd]_{e_n'}  &&  X^I \ar[dd]^{e_n'} \ar[ll]_{\phi_*}  \\\\
  (\psi^n)^{-1}U=V   \ar@{^{(}->}[r] \ar@{..>}[ruu]_{s'} &Y \times Y \times \cdots \times Y  \ar[rr]_{\psi\times \psi \cdots \times \psi} && X\times X \times \cdots \times X & U  \ar@{_{(}->}[l] \ar[uul]_{s}
}$$
Define $s': V\to Y^I$ as $s'(y_1,\cdots, y_n) = \phi_* \circ s\circ\psi^n(y_1,\cdots, y_n)$ where $\phi_* : X^I \to Y^I$ is the map induced by $\phi$. Then   $e_n' \circ s': V \to Y^n$ is the map $ (y_1,\cdots, y_n) \mapsto (\phi\circ\psi(y_1), \cdots , \phi\circ\psi(y_n))$ which is $G$-homotopic to identity.

\end{proof}

\begin{corollary}
For a $G$-connected space $Y$ with $Y^G \neq \phi$, then $Y$ is $G$-contractible if and only if $\tcg{n}{G}{Y} = 1$, for some $n \geq 2$.
\end{corollary}
\begin{proof}
First assume that $Y$ is $G$-contractible. Since $Y$ is $G$-connected and $Y^G \neq \phi$, using  Lemma \ref{lemma} we can say that $Y$ is $G$-homotopy equivalent to a point. Hence the corollary follows from the above proposition.  

Conversely, let $\tcg{n}{G}{Y} = 1$ for some $n \geq 2$. Then there is a $G$-equivariant section $s : Y^n \to Y^I$ of $e'_n$. Fix $y_0 \in Y^G$. Define a homotopy $H : Y \times I \to Y$ by $(y, t) \to s(y, y_0, y_0, \cdots y_0)(t)$. Clearly $H$ is a $G$-homotopy between identity map on $Y$ and $C_{y_0}$(constant map on the orbit $\mathcal{O}(y_0)$). So $Y$ is $G$-contractible.

\end{proof}






\noindent We now give some inequalities relating higher equivariant complexity $\tce{n}{G}{Y}$ of \cite{sarkar} and our $\tcg{n}{G}{Y}.$
 
\begin{prop}\label{prop}
Let $Y$ be a $G$-space. Then the following holds.
\begin{enumerate}[(a)]
\item  For any $n\geq 2$, we have $\tce{n}{G}{Y} \leq \tcg{n}{G}{Y}$.
\item If $H$ and $K$ are subgroups of $G$ such that $Y^H$ is $K$-invariant, then 
$$\tce{n}{K}{Y^H} \leq \tcg{n}{K}{Y^H}  \leq \tcg{n}{G}{Y}.$$
In particular,
$$ \TC_n(Y^H) \leq \tcg{n}{G}{Y}, \, ~~~~\TC_n(Y) \leq \tce{n}{K}{Y} \leq \tcg{n}{G}{Y}.$$

\item If $Y$ is not $G$-connected, then $\tcg{n}{G}{Y} = \infty$ for all $n\geq 2$.

\end{enumerate}
\end{prop}

\begin{proof}
\begin{enumerate}[(a)]

\item This simply follows from the definitions. 

\item Let us prove the inequality $\tcg{n}{K}{Y^H}  \leq \tcg{n}{G}{Y}$. Let $U$ be an $G^n$-invariant open set of $Y^n$ and $s : U \to Y^I$ be a $G$-equivariant section for $e'_n$. Define $V = U \cap (Y^H)^n$, is $K^n$-invariant as $U$ and $(Y^H)^n$ both are $K^n$-invariant. If we restrict the map $s$ on $V$, then for any $(y_1, \cdots, y_n) \in V$ and $h\in H$, $$h.s_{|_V}(y_1, \cdots, y_n) = s_{|_V}(hy_1, \cdots, hy_n) = s_{|_V}(y_1, \cdots, y_n),$$ i.e. the path $s_{|_V}(y_1, \cdots, y_n)$ is in $(Y^H)^I$. It is also clear that $s_{|_V}$ is $K$-equivariant. So $ s_{|_V} : V \to (Y^H)^I$ is a $K$-equivariant section for $e'_{n{|_{(Y^H)^I}}} : (Y^H)^I \to (Y^H)^n$.

\noindent The others inequality follows from combining (a) with results of \cite[Proposition 3.14]{sarkar}.

\item If $Y^H$ is not connected then, $\TC_n(Y^H) = \infty.$ So this follows from part (b).

\end{enumerate}
\end{proof}
\begin{remark}
We can visualise the different inequalities of the above Proposition \ref{prop} the following picture, with an arrow goes from bigger to smaller number. 


$$
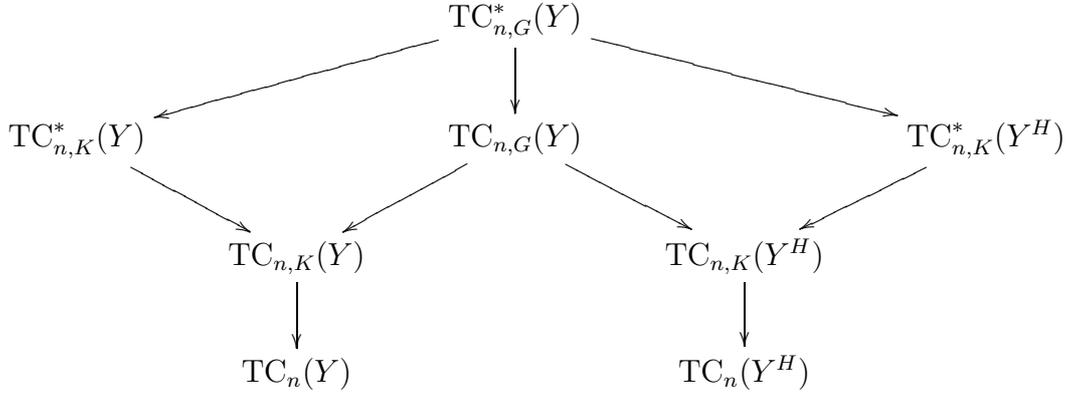
\xymatrix{
&& \tcg{n}{G}{Y} \ar[dll] \ar[d]  \ar[drr] \\
\tcg{n}{K}{Y} \ar[dr] &&  \tce{n}{G}{Y} \ar[dl] \ar[dr] && \tcg{n}{K}{Y^H}   \ar[dl] \\  & \tce{n}{K}{Y} \ar[d] && \tce{n}{K}{Y^H}  \ar[d] \\ & \TC_n(Y) && \TC_n(Y^H)
}$$
\end{remark}

\begin{lemma}\label{ldof}
We have $\tcg{n}{G}{Y}\leq k$ if and only if there exist an open cover of $Y^n$ containing $k$ open sets such that each open set is $G^n$-invariant and $G$-equivariantly deformable into $\Delta(Y) \subset Y^n$.
\end{lemma}

\begin{proof}

It is enough to prove that, a $G$-equivariant section $s : U \to  Y^{I_n}$ exist for $e_n : Y^{I_n} \to Y^n$ on some $G^n$-invariant open subset $U\subset Y^n$ if and only if there is a $G$-homotopy $H = (H_1, H_2, \cdots, H_n) : U \times I \to Y^n$ with $H((y_1, y_2, \cdots, y_n), 0) \in \Delta(Y)$ and   $H((y_1, y_2, \cdots, y_n), 1) = (y_1, y_2, \cdots, y_n)$. We set $$H_j((y_1, \cdots, y_j, \cdots, y_n), t_j) = s(y_1, \cdots, y_j, \cdots, y_n)(t_j),$$ where $t_j\in [0,1]_j \simeq I$ and $j = 1, \cdots, n$. Then existence of one of $s$ and $H$ implies the other one. Hence the Lemma follows.

\end{proof}

\begin{prop}
Assume that $Y$ is a $G$-connected space with $Y^G \neq \phi$. Consider $Y^m$ with product $G^m$-action. Then $$\cat_{G^{n-1}}(Y^{n-1}) \leq \tcg{n}{G}{Y} \leq \cat_{G^n}(Y^n)\leq n\cat_{G}(Y) - (n-1), ~n \geq 2.$$
\end{prop}

\begin{proof}

For the first inequality, consider the pull-back square: 

$$\xymatrix{
P_0 Y \ar[dd]_{p}  \ar@{^{(}->}[rr]^{\inc}  &&  Y^I \ar[dd]^{e'_{n}} \\\\
Y^{n-1}  \ar@{^{(}->}[rr]_{\inc} && Y^{n}
}$$
Here $P_0 Y$ is the set of all paths in $Y$ starting at a base point $y_0\in Y^G$ and $Y^{n-1} \hookrightarrow Y^n,~~ y\mapsto (y_0,y).$ Then for a section of $e'_n$ over $U \subseteq Y^n$, we get a section over $V = (\inc)^{-1} U \subseteq Y^{n-1}.$

For the second inequality, take $U\subseteq Y^n$ be $G^n$-categorical with respect to product action. Then there is a $G^n$-homotopy $H: U\times I \to Y^n$ such that $H(a, 0) =a$ and $H(a,1) = (y_0, \cdots, y_0)$ where $y_0 \in Y^G.$ Then $U$ is $G$-equivariantly deformable to $\Delta(Y).$ 

For the third inequality,  we note that $\cat_{G^n}(Y^n)\leq n\cat_{G}(Y) - (n-1)$ by Lemma \ref{lcatg}. 
\end{proof}

\begin{example}

 Let $\ZZ_2$ acts on $S^m$ by reflection. If $m=1$, then the fixed point set is disconnected. So in this case $\tcg{n}{\ZZ_2}{S^1} = \infty$ for all $n\geq 2.$  For $m\geq 2$, the fixed point set is $S^{m-1}$ which is path connected. By Example \ref{ecat} $\cat_{\ZZ_2}(S^m) =2$. In this case $$\tcg{n}{\ZZ_2}{S^m} \leq n \cat_{\ZZ_2}(S^m) - (n-1) = n.2 - (n-1) = n+1.$$  Also from Proposition \ref{prop}, we get $$\tcg{n}{\ZZ_2}{S^m} \geq \TC_n(  (S^m)^{\ZZ_2}) =\TC_n(S^{m-1}), ~~\tcg{n}{\ZZ_2}{S^m} \geq \TC_n(S^m).$$
 We know from \cite[Section 4]{rud} that $ \TC_n( S^{m}) = n$ if $m$ is odd and $ \TC_n( S^{m}) = n+1$ if $m$ is even. Thus $$\tcg{n}{\ZZ_2}{S^m} = n+1, \mbox{ for all } m,n \geq 2.$$
 Note that $\TC_n(S^m/ \ZZ_2) = \TC_n(D^m)=1.$

\end{example}

\begin{example}
Let $S^3 = \{(z, w)\in \CC^2 |~~ |z|^2 +|w|^2 =1\}. $ Consider the $S^1$-action, defined by $\alpha.(z,w)= (\alpha z,w).$ Then the fixed point sets are
$\{(0, w): ~ w\in S^1\} \cong S^1$.  So $ \tcg{n}{S^1}{S^3} \leq n\cat_{S^1}(S^3) - (n-1).$ But $\cat_{S^1}(S^3) = 2$ (cf. \cite[Example 3.20]{sarkar}).  Therefore $$ n \leq \tcg{n}{S^1}{S^3} \leq  n+1.$$

\end{example}

\noindent We have the following Theorem showing that the strongly equivarint complexity of the universal $\tcg{n}{\pi}{\tX}$ is same as $\tTC{n}{X}$ where $\pi$ is the fundamental group of $X$. This is a generalisation of (\cite[Proposition 3.8]{farber}). The proof is similar.

\begin{theorem}\label{ttcg}

For any locally finite CW complex $X$, we have $$\tTC{n}{X} = \tcg{n}{\pi}{\tX},$$ where $P : \tX \to X$ be the universal covering and $\pi = \pi_1(X)$.

\end{theorem}

\begin{proof}

 We first show $\tTC{n}{X} \leq  \tcg{n}{\pi}{\tX}.$   Let $e_n' \colon \tX^I \to \tX^n$ be the map as above. Assume that $\tU \subset \tX^n$ be an $\pi^n$ invariant open set such that there is a $\pi$-equivariant section $\tilde{s} : \tU \to \tX^I$ of $e_n'$.  Consider the open set $V= \tU/\pi\subset  \xpi $ where $\pi$-acts diagonally. We have the following commutative diagram.

$$
\xymatrix{
\tX^I \ar[dd]_{/\pi}  \ar@/^{2.5pc}/[rrrr]\sp{e_n'}   &&  \tU  \ar[ll]^{\widetilde{s}}\ar[dd]^{/\pi} \ar@{^{(}->}[rr]^{\inc} &&  \tX^n \ar[dd]^{/\pi} \ar[ddrr]^{/\pi^n}  \\\\
X^I   \ar@/_{2.5pc}/[rrrr]\sb{p}  && V   \ar[ll]^{s} \ar@{^{(}->}[rr]_{\inc} &&  \xpi \ar[rr]^{q} && X^n
}$$
Since $\tilde{s}$ were $\pi$-equivariant, the section $s$ exists.  Note that $V= q^{-1}(\tU/\pi^n).$

To prove the other inequality $\tTC{n}{X} \geq  \tcg{n}{\pi}{\tX},$ it is enough to show that given a section $s$ as above, it can be lifted to a section $\widetilde{s}.$  Since the $\pi$-action on the top rows are free, the vertical maps are principle $\pi$-bundles. Consider the classifying maps $\xi: X^I\to B\pi$ and  $\xi':  \xpi \to B\pi$.  Then the classifying map for the $\tU$-bundle is $\xi'\circ\inc$. The existence of $\widetilde{s}$ follows from the following fact of principal bundles: Let $E\to B$ and $E'\to B'$ be two principle $G$-bundles.  Then a map $f: B'\to B$ can be lifted to a bundle map $\tilde{f}: E'\to E$ if and only if $\xi\circ f \simeq \xi' $ where $\xi : B \to BG$ and $\xi': B'\to BG$ are classifying maps of the respective principle $G$-bundles. We apply it to the $\pi$-bundles of the left square of the above diagram. The existence of the bundle map $e'_n$ covering $p$ implies $\xi \simeq \xi' \circ p$.  Now note that $\xi \circ s \simeq \xi' \circ p\circ s = \xi'$. Therefore, by the above fact, the bundle map $\widetilde{s}$ exists making the diagram commutative.

\end{proof}

We now use the strongly equivariant complexity to give an upper bound for higher complexity of total space of a fiber bundle. The following theorem is a generalisation of \cite[Theorem 3.1]{alex}. 

\begin{theorem}\label{tfhtc}
Let $E, B$ be two locally compact metric spaces and $p : E\to B$ be a fiber bundle with fiber $F$ and structure group $G$ acting properly on $F$. Then 
$$\TC_n(E)\leq \TC_n(B)+ \tcg{n}{G}{F} -1.$$
\end{theorem}

\begin{proof}

Let $\TC_n(B) = r, \tcg{n}{G}{F} = r'$ and $k = r + r' - 1$ such that $r, r', k \geq 1$. Consider $F^n$ as a $G$-space with diagonal action of $G$. Then there is a $\Delta(B)$-deformable open cover $\{U_1, U_2, \cdots, U_r\}$ of $B^n$ and a $G$-equivariantly deformable into $\Delta(F)$ open cover $\{V_1, V_2, \cdots, V_{r'}\}$ of $F^n$ by $G^n$-invariant sets, by Lemma \ref{ldof}. Using Proposition \ref{propdeform} we can extend the above open covers to an $r$-cover  $\{U_1, U_2, \cdots, U_r, \cdots, U_k\}$ and an $r'$-cover $\{V_1, V_2, \cdots, V_{r'}, \cdots, V_k\}$ for $B^n$ and $F^n$ respectively, with the same property.

Consider the universal $F^n$-bundle $q : F^{n} \times_{G^n} E(G^n) \to B(G^n)$ and classifying map $g : B^n \to B(G^n)$ for the $F^n$-bundle $p^n : E^n \to B^n$. Set $O_i = V_i \times_{G^n} E(G^n), \ i = 1, 2, \cdots, k$ so that $\{O_i\}_{i=1}^k$ is an $r'$-cover of $ {F^n \times_{G^n} E(G^n)}$ and $W_i = O'_i \cap (p^n)^{-1}(U_i)$, where $ O'_i = (g')^{-1}(O_i) $. Using Lemma \ref{lemmacover} we can say that $\{W_i\}_{i=1}^k$ covers $E^n$. Now we show that the each set $W_i$ is deformable to $\Delta(E)$, in two steps.

\noindent \textit{Step-I:} Consider the composition map, $W_i \times I \xrightarrow{p^n \times \Id} U_i \times I \xrightarrow{H} B^n $, where $H$ is a deformation of $U_i$ into $\Delta(B)$. Using the homotopy lifting property of the fibre bundle $f=p^{n}_{|_{O'_i}} : O'_i \to B^n$, we can say that $W_i$ can be deformed in $O'_i$ to the preimage $f^{-1}(\Delta(B))$.

$$\xymatrix{
W_i \times \{0\} \ar@{^{(}->}[dd] \ar@{^{(}->}[rr]  &&  O'_i \ar@{^{(}->}[rr] \ar[dd]^{p^{n}_{|_{O'_i}}=f}  &&  E^n  \ar[dd]^{p^n = g^*(q)}  \ar[rr]^{g'}&&  F^{n} \times_{G^n} E(G^n)  \ar[dd]^{q}  \\\\
W_i \times I \ar@{..>}[rruu]   \ar[rr]_{H \circ (p^n\times \Id)} && B^n \ar[rr]_{=} && B^n  \ar[rr]_{g}  &&  B(G^n)}$$

\noindent \textit{Step-II:} Now we show that $f^{-1}(\Delta(B))\subseteq O'_i$ can be deformed into $\Delta(E)$. Let $\phi^i_t : V_i \to F^n$ be an $G$-equivariant deformation of $V_i$ into $\Delta(F)$. It defines a deformation of $V_i \times_G EG$ to $\Delta(F) \times_G EG$ in $F^n \times_G EG$. Observe that the bundle $q : F^{n} \times_{G^n} E(G^n) \to B(G^n)$  restricted over $\Delta(BG) \cong BG$ is $ F^{n} \times_{G} EG \to BG$ with the diagonal action of $G$ on $F^n$. Then the above deformation defines a fiberwise deformation of $O_i$ over $\Delta(BG)$ into $\Delta(F\times_G EG)$. This will induce a fiberwise deformation of $O'_i$ over $\Delta(B)$, i.e.  of $f^{-1}(\Delta(B))$, into $\Delta(E)$.

The concatenation of the above two deformation in Step-I and Step-II defines a deformation of $W_i$ into $\Delta(E)$.
\end{proof}

\begin{remark} Following the arguments of \cite[Theorem 3.3]{alex}, the Corollary \ref{cbound} can also be deduced using the above Theorem \ref{tfhtc}.   Let $\tX$ denote the universal cover of $X$. Consider the fiber bundle $\tX\times_{\pi} E\pi \to B\pi.$ It has fiber $\tX$ and structure group $\pi$. Applying Theorem \ref{tfhtc} to this bundle we get the following inequality 
 $$\TC_n(\tX\times_{\pi} E\pi) \leq \tcg{n}{\pi}{\tX} + \TC_n(B\pi)-1.$$  Now by Theorem \ref{ttcg}   $\tcg{n}{\pi}{\tX} = \tTC{n}{X}$.  Also from the Equation \ref{eqttc}, with $k=1$,  we have $\tTC{n}{X}\leq  \ceil[\Bigg]{\frac{n  \dim X - 1}{2}} + 1.$ Putting it in the above inequality 
 
 $$\TC_n(\tX\times_{\pi} E\pi)\leq \TC_n(\pi) + \ceil[\Bigg]{\frac{n  \dim X - 1}{2}} .$$
Note that the map induced by covering projection $\tX\times_{\pi} E\pi \to X$ is a homotopy equivalent, since it has contractible fiber $E\pi$. So $\TC_n(\tX\times_{\pi} E\pi)  = \TC_n(X). $

\end{remark}

\end{section}


\begin{thebibliography}{12}

\bibitem{AngC} A.~ Ángel, H.~Colman,
``Equivariant topological complexities'',  {Topological complexity and related topics,
Contemp. Math.},  \textbf{702}  (2018), pp.~1–15.

\bibitem{cf} A.~ Costa, M. ~Farber,
``Motion planning in spaces with small fundamental groups",
\textit{ Commun. Contemp. Math.},  \textbf{12} (1) (2010), pp.~107–119.

\bibitem{alex} A.~ Dranishnikov,
``On topological complexity of twisted products",
\textit{Topology Appl.} \textbf{179} (2015), pp.~ 74- 80. 

\bibitem{kolmo} A. N.~ Kolmogorov, 
``On the representation of continuous functions of many variables by superposition of continuous functions of one variable and addition", \textit{ Amer. Math. Soc.}  Transl. (2) \textbf{28} (1963) , pp.~ 55–59.

\bibitem{sva} A. ~ \v{S}varc,
``The genus of a fiber space'', 
\textit{Amer. Math. Soc. Transl. Ser. 2}, \textbf{55} (1966), pp.~ 49-140.

\bibitem{Col} H.~Colman,
``Equivariant LS-category for finite group actions'', 
Lusternik-Schnirelmann category and related topics, Contemp. Math. 316, Amer. Math. Soc. (2002) pp.~35-40.

\bibitem{ColG} H.~ Colman, M.~ Grant,
``Equivariant topological complexity",  
\textit{Algebr. Geom. Topol.} \textbf{12} (2012), no. 4, pp.~2299 - 2316.

\bibitem{sarkar1} M.~ Bayeh, S.~ Sarkar,
``Some aspects of equivariant LS-category'',
\textit{Topology Appl.} \textbf{196}  (2015), pp.~ 133- 154. 

\bibitem{sarkar} M.~ Bayeh, S.~ Sarkar,
 ``Higher equivariant and invariant topological complexity", arXiv:1804.08006v1.

\bibitem{far} M.~ Farber,
``Topological complexity of motion planning", 
\textit{Discrete Comput. Geom.} 
\textbf{29} (2003), no. 2, pp.~211-221.

\bibitem{farb} M.~ Farber, J.~ Oprea,
``Higher topological complexity of aspherical spaces", 
\textit{Topology Appl.} \textbf{258} (2019), pp.~142-160.

\bibitem{farber} M.~ Farber, M.~ Grant, G.~ Lupton, J.~ Oprea,
``An upper bound for topological complexity",
\textit{Topology Appl.} \textbf{255} (2019), pp. ~109 -125. 

\bibitem{farberm} M.~ Farber, M.~ Grant, G.~ Lupton, J.~ Oprea,
``Bredon cohomology and robot motion planning",
\textit{Algebr. Geom. Topol.} \textbf{19} (2019), no. 4, pp. ~2023 -2059. 

\bibitem{cor} O.~Cornea, G.~Lupton, J.~Oprea, D.~Tanr\'{e},
``Lusternik-Schnirelmann category, Math. Surveys and Monographs'' 103, Amer. Math. Soc. (2003).

\bibitem{ostrand} P. A. ~ Ostrand, 
``Dimension of metric spaces and Hilbert's problem 13",
\textit{Bull. Amer. Math. Soc.} \textbf{71} (1965),  pp. ~619 -622.

\bibitem{Mar} W.~Marzantowicz,
``A G-Lusternik -Schnirelman category of space with an action of a compact Lie group'', \textit{Topology} \textbf{28} (1989) pp.~403-412.

\bibitem{LubM} W.~ Lubawski,  W.~ Marzantowicz,
``Invariant topological complexity",
 \textit{Bull. London Math. Soc.} \textbf{47} (2014), pp.~~101-117.

\bibitem{rud} Y.B.~Rudyak,
``On higher analogs of topological complexity", \textit{Topology Appl.} \textbf{157} (2010), no. 5,  pp.~916-920.

\bibitem{BlaK} Z.~Blaszczyk, M.~Kaluba, 
``On equivariant and invariant topological complexity of smooth $\ZZ_p$-spheres",
 \textit{Proc. Amer. Math. Soc.} \textbf{145}(2017), pp.~~ 4075 - 4086.

\end{thebibliography}
\end{document}